\documentclass[a4paper,12pt]{article}

\usepackage{url}
\usepackage{tikz}
\usepackage{graphicx}
\usepackage{amsfonts}
\usepackage{amssymb}
\usepackage{amsmath}
\usepackage{amsthm}
\usepackage[english]{babel}
\usepackage{ctex}
\usepackage{hyperref}
\usepackage{cases}
\usepackage{authblk}
\usepackage{mathtools}
\usepackage{pifont}
\usepackage[numbers,sort&compress]{natbib}

\usepackage{enumerate}
\usepackage[english]{babel}

\hypersetup
{
	colorlinks=true,
	linkcolor=blue,
	filecolor=blue,
	urlcolor=blue,
	citecolor=cyan,
}

\usepackage[nameinlink]{cleveref}

\newtheorem{lemma}{Lemma}[section]
\Crefname{lemma}{Lemma}{Lemmas}
\newtheorem{theorem}[lemma]{Theorem}
\Crefname{theorem}{Theorem}{Theorems}
\newtheorem{cor}[lemma]{Corollary}

\newtheorem{claim}[]{\indent Claim}[section]

\Crefname{section}{Section}{Sections}

\newcounter{tbox}

\begin{document}

\title{Cycles and paths through specified vertices in graphs with a given clique number}

\author[1]{Chengli Li\footnote{ Email: lichengli0130@126.com.}}
\author[2]{Leyou Xu\footnote{Email: leyouxu@m.scnu.edu.cn}}
	
\affil[1]{\small Department of Mathematics, East China Normal University, Shanghai, 200241, China}
\affil[2]{\small School of Mathematical Sciences, South China Normal University, Guangzhou, 510631, China}

\date{}
\maketitle

\begin{abstract}

B. Bollob\'{a}s and G. Brightwell and
independently R. Shi proved the existence of a cycle through all vertices whose degrees at least $\frac{n}{2}$ in any $2$-connected graph of order $n$. Motivated by this result, we prove the existence of a cycle through all vertices whose degrees at least $n-\omega$ in any $2$-connected graph $G$ of order $n$ with clique number $\omega$ unless $G$ is a specific graph. Moreover, we show that for any pair of vertices whose degrees are at least $n-\omega+1$ in a graph $G$ of order $n$ with clique number $\omega$, there exists a path joining them which contains all vertices of degree at least $n-\omega+1$ unless $G$ belongs to certain graph classes. In doing so, we prove the existence of a $(u,v)$-path through all vertices whose degrees at least $\frac{n+1}{2}$ in any graph of order $n$, where $u,v$ are two distinct vertices of degree at least $\frac{n+1}{2}$.
\\
{\bf Keywords: }cycle, path, clique number, minimum degree \\
{\bf Mathematics Subject Classifications:} 05C38, 05C69
\end{abstract}

\section{Introduction}

We only consider finite undirected graphs without loops and multiedges in this paper. 

A Hamilton cycle (path, respectively) of a graph $G$ is a cycle (path, respectively) containing all vertices of $G$. If $G$ contains a Hamilton cycle, then $G$ is said to be hamiltonian. 
Ever since Dirac \cite{Dirac} proved his renowned theorem in 1952, stating that a graph with minimum degree at least $\frac{n}{2}$ is hamiltonian, a great number of related results have emerged. Moreover, the study of hamiltonian problems has become a hot topic in graph theory from then on. For surveys, we refer the reader to \cite{Broersma,Gould,Li}.
Based on Dirac's theorem, Bollobás and Brightwell, as well as Shi independently, discovered the following interesting theorem.

\begin{theorem}{\textup {(Bollob\'{a}s and Brightwell \cite{Bollobas}, Shi \cite{Shi})}}\label{th1}
In a $2$-connected graph of order $n$, there exists a cycle containing all vertices of degree at least $\frac{n}{2}$.
\end{theorem}

The clique number of a graph $G$ is the maximum cardinality of a clique in $G$.
In this paper, we investigate the exsistence of cycles and paths containing specified vertices, concerning graphs with given order and clique number.
Our main results are as follows.

\begin{theorem}\label{x}
In a $2$-connected graph $G$ of order $n$ with clique number $\omega$, there exists a cycle containing all vertices of degree at least $n-\omega$, unless $G\cong K_{n-\omega}\vee (K_{2\omega-n}\cup \overline{K_{n-\omega}})$ with $\frac{n+1}{2}\le\omega\le n-2$. 
\end{theorem}

One sufficent condition for hamiltonian graph follows by Theorem~\ref{x} immediately, which can also be obtianed by a result of Yuan \cite{Yuan}.

\begin{cor}\label{cor1}
Let $G$ be a $2$-connected graph $G$ of order $n$ with minimum degree $\delta$ and clique number $\omega$. If $\delta+\omega\ge n$, then $G$ is hamiltonian unless $G\cong K_{n-\omega}\vee (K_{2\omega-n}\cup \overline{K_{n-\omega}})$ with $\frac{n+1}{2}\le\omega\le n-2$. 
\end{cor}

A $k$-cycle is a cycle of length $k$. In 1971, Bondy~\cite{Bondy} introduced the concept of a pancyclic graph. A graph $G$ of order $n$ is called pancyclic if for every integer $k$ with $3\le k\le n$, $G$ contains a $k$-cycle. Based on Theorem \ref{x}, we can obtain a result stronger than Corollary \ref{cor1}.

\begin{theorem}\label{y}
Let $G$ be a $2$-connected graph $G$ of order $n$ with minimum degree $\delta$ and clique number $\omega$. If $\delta+\omega\ge n$, then $G$ is pancyclic unless $G\cong K_{2,2}$ or $G\cong K_{n-\omega}\vee (K_{2\omega-n}\cup \overline{K_{n-\omega}})$ with $\frac{n+1}{2}\le \omega\le n-2$.
\end{theorem}

A $(u,v)$-path in a graph is a path with endpoints $u$ and $v$. Motivated by Theorems \ref{th1} regarding the existence of cycles, we focus on the existence of paths through specified vertices, see Theorem~\ref{Theorem-orepath}. By means of this theorem, we can prove the path version of Theorem \ref{x}, see Theorem~\ref{Theorem-path}. 

\begin{theorem}\label{Theorem-orepath}
Let $G$ be a graph of order $n$ and let $u,v$ be two vertices of degree at least $\frac{n+1}{2}$. Then there exists a $(u,v)$-path of $G$ which contains all vertices of degree at least $\frac{n+1}{2}$.
\end{theorem}

For two graphs $G$ and $H$, we write $G\cup H$ for the disjoint union of $G$ and $H$ and $G\vee H$ for the graph obtained from $G\cup H$ by adding all possible edges between $G$ and $H$.
Let $n$ and $\omega$ be integers with $\frac{n+2}{2}\le \omega\le n-1$. Let $H_1=K_2\vee (K_{\omega-2}\cup K_1)$ and $H_2=K_{1,n-\omega-1}$. Denote by $x$ the vertex of degree two in $H_1$ and by $y$ the vertex of degree $n-\omega-1$ in $H_2$. The coalescence of $H_1$ and $H_2$ at $x$ and $y$, denoted by $H_1xyH_2$, is the graph obtained by identifying $x$ and $y$ from $H_1$ and $H_2$, see Figure~\ref{fig}.
Let $\mathcal{H}(n,\omega)$ be the set of graphs containing $H_1xyH_2$ as a spanning subgraph, which are spanning subgraphs of $K_2\vee (K_{\omega-2}\cup K_{n-\omega})$. Evidently, $H_1xyH_2,K_2\vee (K_{\omega-2}\cup K_{n-\omega})\in \mathcal{H}(n,\omega)$. 

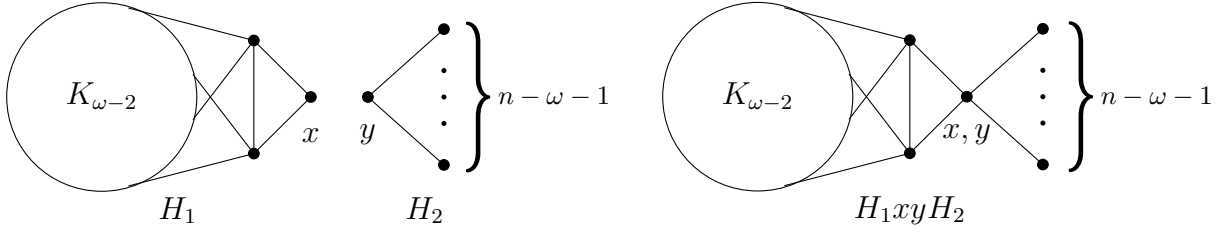
\begin{figure}[htbp]\label{fig}
\centering
\begin{minipage}{0.485\linewidth}
\begin{tikzpicture}
\draw [black](-2,0) circle (1.25);
\node at (-2,0) {$K_{\omega-2}$};
\filldraw [black] (0,0.75) circle (2pt);
\filldraw [black] (0,-0.75) circle (2pt);
\filldraw [black] (0.75,0) circle (2pt);
\draw  [black](0,0.75)--(0.75,0)--(0,-0.75)--(0,0.75);
\draw  [black](-0.8,-0.3)--(0,0.75)--(-1.65,1.18);
\draw  [black](-0.8,0.3)--(0,-0.75)--(-1.65,-1.18);
\node at (0.75,-0.5) {$x$};
\node at (-1,-1.5) {$H_1$};
\filldraw [black] (1.5,0) circle (2pt);
\node at (1.5,-0.5) {$y$};
\filldraw [black] (2.5,0.9) circle (2pt);
\filldraw [black] (2.5,-0.9) circle (2pt);
\filldraw [black] (2.5,0.35) circle (0.75pt);
\filldraw [black] (2.5,0) circle (0.75pt);
\filldraw [black] (2.5,-0.35) circle (0.75pt);
\draw [black](2.5,0.9)--(1.5,0)--(2.5,-0.9);
\node at (2.95,0) {\resizebox{0.45cm}{1.5cm}{$\}$}};
\node at (3.95,0) {\footnotesize $n-\omega-1$};
\node at (2.25,-1.5) {$H_2$};
\end{tikzpicture}
\end{minipage}
\hspace{0.25cm}
\begin{minipage}{0.485\linewidth}
\begin{tikzpicture}
\draw [black](-2,0) circle (1.25);
\node at (-2,0) {$K_{\omega-2}$};
\filldraw [black] (0,0.75) circle (2pt);
\filldraw [black] (0,-0.75) circle (2pt);
\filldraw [black] (0.75,0) circle (2pt);
\filldraw [black] (1.75,0.9) circle (2pt);
\filldraw [black] (1.75,-0.9) circle (2pt);
\draw  [black](0,0.75)--(0.75,0)--(0,-0.75)--(0,0.75);
\draw  [black](-0.8,-0.3)--(0,0.75)--(-1.65,1.18);
\draw  [black](-0.8,0.3)--(0,-0.75)--(-1.65,-1.18);
\draw  [black](1.75,0.9)--(0.75,0)--(1.75,-0.9);
\filldraw [black] (1.75,0.35) circle (0.75pt);
\filldraw [black] (1.75,0) circle (0.75pt);
\filldraw [black] (1.75,-0.35) circle (0.75pt);
\node at (0.75,-0.5) {$x,y$};
\node at (2.25,0) {\resizebox{0.45cm}{1.5cm}{$\}$}};
\node at (3.25,0) {\footnotesize $n-\omega-1$};
\node at (0,-1.5) {$H_1xyH_2$};
\end{tikzpicture}
\end{minipage}
\caption{The graph $H_1$, $H_2$ and $H_1xyH_2$ (from left to right).}
\end{figure}

\begin{theorem}\label{Theorem-path}
Let $G$ be a graph of order $n$ with clique number $\omega$. If $u$ and $v$ are two vertices of degree at least $n-\omega+1$, then there exists a $(u,v)$-path containing all vertices of degree at least $n-\omega+1$, unless $\frac{n+2}{2}\le \omega\le n-1$, and $G\cong K_{n-\omega+1}\vee (K_{2\omega-n-1}\cup \overline{K_{n-\omega}})$  or $G\in\mathcal{H}(n,\omega)$. 
\end{theorem}

A graph is called hamiltonian-connected if there is a Hamilton path between any two distinct vertices. 
From Theorem \ref{Theorem-path}, we can obtain a sufficient condition for hamiltonian-connected graph directly.
 
\begin{cor}\label{l}
Let $G$ be a graph of order $n$  with minimum degree $\delta$ and clique number $\omega$. If $\delta+\omega\ge n+1$, then $G$ is hamiltonian-connected unless $\frac{n+2}{2}\le \omega\le n-1$, and $G\cong K_{n-\omega+1}\vee (K_{2\omega-n-1}\cup \overline{K_{n-\omega}})$  or $G\in\mathcal{H}(n,\omega)$. 
\end{cor}

\section{Notations}

Let $G$ be a graph with vertex set $V(G)$ and edge set $E(G)$. Denote by $\overline{G}$ the completement of $G$.
For $v\in V(G)$, $N_G(v)$ denotes the neighborhood of $v$ in $G$ and $d_G(v)$ the degree of $v$, i.e., $d_G(v)=|N_G(v)|$. 
We write $N(v)$ and $d(v)$ for $N_G(v)$ and $d_G(v)$ whenever there are no confusions.
Denote by $\delta(G)$ the minimum degree of $G$.
For $v\in V(G)$ and a subgraph $F$ of $G$, let
$N_F(v)=N_G(v)\cap V(F)$. 

For a graph $G$ with $S\subset V(G)$, denote by $G[S]$ the subgraph of $G$ induced by $S$. Let $G-S=G[V(G)\setminus S]$. 
Denote by $K_n$ the complete graph of order $n$ and $K_{a,b}$ the complete bipartite graph with partite size $a$ and $b$. 

Let $C$ be a cycle with vertices labeled clockwise. For $u\in V(C)$, denote by $u^{+}$ and $u^-$ the immediate successor and predecessor on $C$. 
For $S\subseteq V(C)$, let $S^+=\{u^+:u\in S\}$ and $S^-=\{u^-:u\in S\}$.
For $u,v\in V(C)$, denoted by $u\overrightarrow{C}v$ the segment of $C$ from $u$ to $v$ clockwise.
Denoted by $u\overleftarrow{C} v$ the opposite segment of $C$ from $u$ to $v$.

For a path $P$ with $u,v\in V(P)$, denote by $uPv$ the segment of $P$ from $u$ to $v$. We usually consider the $(u,v)$-path $P$ as the path from $u$ to $v$, i.e., $uPv$. 

Let $P$ be a $(u,v)$-path. For $x\in V(P)$ with $x\ne v$, denote by $x^+$ the immediate successor on $P$.  For $x\in V(P)$  with $x\ne u$, denote by $x^-$ the predecessor on $P$. 
For $S\subseteq V(P)$, let $S^+=\{x^+:x\in S\setminus\{v\}\}$ and $S^-=\{x^-:x\in S\setminus\{u\}\}$. Obviously, $|S^+|=|S|$ or $|S^+|=|S|-1$. 

\section{Proofs of Theorems \ref{x} and \ref{y} about cycles}

\begin{proof}[Proof of Theorem \ref{x}]
If $n-\omega\ge \frac{n}{2}$, then the result follows by Theorem \ref{th1}. Suppose in the following that $n-\omega< \frac{n}{2}$, that is, $\omega\ge \frac{n+1}{2}$. 

It is trivial if $\omega=n$. If $\omega=n-1$, then, as $G$ is $2$-connected, $G$ contains $K_2\vee (K_{n-3}\cup K_1)$ as a spanning subgraph, and so $G$ is hamiltonian, as desired. So we may assume in the following that $\frac{n+1}{2}\le \omega\le n-2$.

Suppose to the contrary that there is no such cycles. Let $C$ be a cycle of $G$ such that $C$ contains vertices of degree at least $n-\omega$ as many as possible. 
Let $R=V(G)\setminus V(C)$ and $K$ be a largest clique in $G$.

\begin{claim}\label{newc1}
$C$ contains at least two vertices of $K$.
\end{claim}
\begin{proof}
As $\omega\ge \frac{n+1}{2}$, $\omega>n-\omega$. Then $|V(C)|\ge\omega$ by the assumption of $C$ and so by principle of inclusion-exclusion, \[
|V(C)\cap K|=|V(C)|+|K|-|V(C)\cup K|\ge \omega+\omega-n\ge 1,
\]
that is, there is at least one vertex of $K$ on $C$, say $u$.

If $u$ is the unique vertex of $K$ on $C$, then $|V(C)\setminus K|\ge \omega-1$ and so \[
n\ge |V(C)\setminus K|+|K|\ge \omega-1+\omega\ge n,
\]
showing that $n$ is odd, $\omega=\frac{n+1}{2}$ and $|V(C)|=\omega=\frac{n+1}{2}$. As $G$ is $2$-connected, there exists some vertex $v\in V(C)\setminus \{u\}$ adjacent to some vertex $z\in K\setminus\{u\}$. Then, with $z_1,\dots,z_{\omega-2}$ being vertices of $K\setminus\{u,z\}$, $uz_1\dots z_{\omega-2}zv\overrightarrow{C}u$ is a cycle containing more vertices of degree at least $n-\omega$, a contradiction. So there are at least two vertices of $K$ on $C$, as desired.
\end{proof}

\begin{claim}\label{newc2}
For any vertex $x$ outside $C$, $N_C^+(x)$ and $N_C^-(x)$ are independent.
\end{claim}
\begin{proof}
To the contrary, assume that $z_1^+z_2^+\in E(G)$ with $z_1,z_2\in N_C(x)$. Then \[
xz_1\overleftarrow{C}z_2^+z_1^+\overrightarrow{C}z_2x
\]
is a cycle containing more vertices of degree at least $n-\omega$ in $G$, a contradiction. This shows that $N_C^+(x)$ is independent. Similarly, $N_C^-(x)$ is independent, as desired.
\end{proof}

\begin{claim}\label{c1}
Each vertex in $K$ is contained in $C$ or $G\cong K_{\frac{n-1}{2}}\vee \overline{K_{\frac{n+1}{2}}}$.
\end{claim}
\begin{proof}
Suppose to the contrary that there exists some vertex in $K$ that is not contained in $C$. Let $\omega_1=|K\cap R|$ and $\omega_2=|K\cap V(C)|$. Evidently, $\omega_1+\omega_2=\omega$ with $\omega_1,\omega_2\ge 1$. Moreover, by Claim \ref{newc1}, we have $\omega_2\ge 2$.

By the assumption of $C$, for any two vertices of $K$ on $C$, the segment between these two vertices has at least $\omega_1$ vertices outside $K$. Then $|V(C)|\ge \omega_1\omega_2+\omega_2$. So 
\[
n\ge |V(C)|+|K\cap R|\ge \omega_1\omega_2+\omega_2+\omega_1\ge 2\omega-1\ge n.
\]
This shows that $n$ is odd, $\omega=\frac{n+1}{2}$, $\omega_1=1$, $\omega_2=\omega-1$, $R=K\setminus V(C)$ and $C$ contains exactly $2\omega_2$ vertices. 
Denote by $x$ the unique vertex outside $C$. It then follows by Claim~\ref{newc2} that $N_C^+(x)$ is independent and so $N_C^+(x)\cap K=\emptyset$. Recall that $C$ contains exatcly $2\omega_2$ vertices.  
Then $V(C)=N_C(x)\cup N_C^+(x)$ and $K=N_C(x)\cup \{x\}$.
We claim that each vertex in $N_C^+(x)$ has degree at least $n-\omega$. To the contrary, suppose that $y\in N_C^+(x)$ has degree at most $n-\omega-1$. Then $y^-\overleftarrow{C}y^+xy^-$ is a cycle containing more vertices of degree at least $n-\omega$ in $G$, a contradiction.
Thus each vertex in $N_C^+(x)$ is adjacent to any vertex in $N_C(x)$ and so $G\cong K_{\frac{n-1}{2}}\vee \overline{K_{\frac{n+1}{2}}}$.
\end{proof}

If $G\cong K_{\frac{n-1}{2}}\vee \overline{K_{\frac{n+1}{2}}}$, then $G$ does not contain a cycle satisfying the condition. Suppose in the following that $G\ncong K_{\frac{n-1}{2}}\vee \overline{K_{\frac{n+1}{2}}}$. 
By Claim \ref{c1}, each vertex in $K$ is contained in $C$. 
Let $x$ be a vertex of degree at least $n-\omega$ outside $C$. 

By Claim~\ref{newc2}, $N_C^+(x)$ is independent. Then $|N_C^+(x)\cap K|\le 1$ and so
\[
n\ge |V(C)|+|N_R(x)|+|\{x\}|\ge |K\setminus N_C^+(x)|+|N_C^+(x)|+|N_R(x)|+1\ge \omega-1+n-\omega+1=n.
\]
As the above inequality holds with equality, we have $|N_C^+(x)\cap K|= 1$, $V(C)=K\cup N_C^+(x)$ and $V(G)=V(C)\cup N_R(x)\cup \{x\}$. Similarly, we have $|N_C^-(x)\cap K|= 1$ and $V(C)=K\cup N_C^-(x)$. It then follows that $N_C^+(x)\setminus K=N_C^-(x)\setminus K$ and so $|N_C^+(x)\cap N_C^-(x)|\ge d_C(x)-1$.

Denote by $x_1,\dots,x_t$ the neighbors of $x$ on $C$ clockwise, where $t=d_C(x)$. Recall that $|N_C^+(x)\cap K|=1$ and $V(C)=K\cup N_C^+(x)$. Assume that $x_1^+,\dots,x_{t-1}^+\in N_C^+(x)\setminus K$. Then $x_1^+,\dots,x_{t-1}^+\in N_C^-(x)\setminus K$, which shows that $x_{i}^+=x_{i+1}^-$ for $i=1,\dots,t-1$. 

If $N_R(x_i^+)\ne \emptyset$ for some $i=1,\dots,t-1$, say $x'\in N_R(x_i^+)$, then, as $xx_i^+\notin E(G)$ and $R=N_R(x)\cup \{x\}$, \[
x_ixx'x_i^+\overrightarrow{C}x_i
\] 
is a cycle with more vertices of degree $n-\omega$, a contradiction. So $N_R(x_i^+)=\emptyset$ for any $i=1,\dots,t-1$.

We claim that $N(x_i^+)\cap (K\setminus\{x_1,\dots,x_t\})=\emptyset$ for $i=1,\dots,t-1$. Otherwise, assume that $x_i^+z\in E(G)$ for some $z\in K\setminus\{x_1,\dots,x_t\}$. Then, as $N_C^+(x)$ and $N_C^-(x)$ are independent,  $z\ne x_j^+,x_j^-$ for any $j=1,\dots,t$ and so $z$ lies in the segment $x_t^+\overrightarrow{C}x_1^-$ with $z\ne x_t^+,x_1^-$. So $z^+\in K$ and 
\[
xx_i\overleftarrow{C}z^+x_t^+\overrightarrow{C}zx_i^+\overrightarrow{C}x_tx
\]
is a cycle with more vertices of degree at least $n-\omega$, a contradiction. 
Recall that $V(G)=V(C)\cup N_R(x)\cup \{x\}$ and $N_R(x_i^+)=\emptyset$ for $i=1,\dots,t-1$. We have $N(x_i^+)\subseteq N_C(x)$ for $i=1,\dots,t-1$. Recall that $x_i^+=x_{i+1}^-$. Then $d(x_i^+)\ge n-\omega$ as otherwise, 
\[
x_ixx_{i+1}\overrightarrow{C}x_i
\]
is a cycle with more vertices of degree $n-\omega$, a contradiction. So
$t=|N_C(x)|\ge d(x_i^+)\ge n-\omega$. It then follows from 
\[
n\ge  |V(C)|+|\{x\}|\ge \omega+t-1+1\ge n
\]
that $t=n-\omega$, $R=\{x\}$ and $N(x_i^+)=N(x)=\{x_1,\dots,x_t\}$. This shows that $G\cong K_{n-\omega}\vee (K_{2\omega-n}\cup \overline{K_{n-\omega}})$. 
\end{proof}


We need some well-known results for the proof of Theorem~\ref{y}. 

\begin{theorem}{\textup {(Bondy \cite{Bondy})}}\label{bondy}
Let $G$ be a hamiltonian graph of order $n$ with $|E(G)|\ge \frac{n^2}{4}$. Then $G$ is either pancyclic or else is the complete bipartite graph $K_{n/2,n/2}.$
\end{theorem}

The following lemma is stated in the proof of Theorem \ref{bondy}, see \cite{Bondy}.

\begin{lemma}{\textup {(Bondy \cite{Bondy})}}\label{Lemma-Bondy}
Let $G$ be a hamiltonian graph of order $n$ with a Hamilton
cycle $u_1u_2\dots u_nu_1$. If $d(u_1)+d(u_n)>n$, then $G$ is pancyclic.
\end{lemma}

\begin{lemma}{\textup {(Faudree et al. \cite{Faudree})}}\label{Lemma-Faudree}
Let $G$ be a hamiltonian graph of order $n$. If $|E(G)|>\left(\frac{n-1}{2}\right)^2-1$, then $G$ is pancyclic or bipartite.
\end{lemma}

\begin{proof}[Proof of Theorem \ref{y}]
If $G\cong K_{n-\omega}\vee (K_{2\omega-n}\cup \overline{K_{n-\omega}})$ with $\frac{n+1}{2}\le \omega\le n-2$, then $G$ is not pancyclic. Suppose in the following that $G\ncong K_{n-\omega}\vee (K_{2\omega-n}\cup \overline{K_{n-\omega}})$ with $\frac{n+1}{2}\le \omega\le n-2$. Then by Corollary~\ref{cor1}, $G$ has a Hamilton cycle, say $u_1u_2\dots u_nu_1$.

If $\delta\ge \frac{n}{2}$, then $|E(G)|\ge \frac{n^2}{4}$ and so by Theorem~\ref{bondy}, $G$ is either pancyclic or else is the complete bipartite graph $K_{n/2,n/2}$. Combining with the condition that $\delta+\omega\ge n$, we have $G$ is pancyclic unless $G\cong K_{2,2}$.
So we assume in the following that $\delta\le \frac{n-1}{2}$. Then $\omega\ge \frac{n+1}{2}$, which implies that $G$ is non-bipartite. Let $K$ be a maximum clique in $G$. Then there exists an integer $i$ such that $u_i,u_{i+1}\in K$ and $u_{i-1}\notin K$, where $i$ modules $n$. Thus $d(u_i)\ge \omega$ and $d(u_{i+1})\ge \omega-1$ and hence
\[
d(u_i)+d(u_{i+1})\ge 2\omega-1\ge n.
\]
If $d(u_i)+d(u_{i+1})>n$, then $G$ is pancyclic by Lemma~\ref{Lemma-Bondy}.
If $d(u_i)+d(u_{i+1})=n$, then $\omega=\frac{n+1}{2}$ and so  $\delta=\frac{n-1}{2}$. Therefore
\[
|E(G)|\ge \frac{n(n-1)}{4}>\left(\frac{n-1}{2}\right)^2-1,
\]
which follows by Lemma~\ref{Lemma-Faudree} that $G$ is pancyclic. This completes the proof.
\end{proof}

\section{Proofs of Theorems \ref{Theorem-orepath} and \ref{Theorem-path} about paths}

We give some additional terminology and notation first.

Let $G$ be a graph on $n$ vertices and $k\ge 2$ be an integer. We call a sequence of vertices $P=v_1v_2\dots v_k$ an ore $(v_1,v_k)$-path of $G$, if for
all $i$ with $1\le i \le k-1$, either $v_iv_{i+1}\in E(G)$ or $d(v_i)+d(v_{i+1})\ge n+1$.
The deficit degree of the ore path $P$ is defined by $\textup{def}(P)=|\{i:v_iv_{i+1}\notin E(G)~\textup{with}~1\le i\le k-1\}|$. Thus a path is an ore path with deficit degree $0$. 
The concept of ore path was raised by Li and Zhang in \cite{LZ}, which was raised with the concept of ore cycle. 

Now, we prove the following lemma on ore path.
\begin{lemma}\label{Lemma-orepath}
Let $G$ be a graph of order $n$ and let $u,v$ be two vertices with $d(u)+d(v)\ge n+1$. If $P$ is an ore $(u,v)$-path of $G$, then there exists a $(u,v)$-path of $G$ which contains all the vertices of $P$.
\end{lemma}

\begin{proof}
Assume to the contray that there is no such $(u,v)$-path. Let $Q$ be an ore $(u,v)$-path that contains all vertices of $P$ such that $\textup{def}(Q)$ is as small as possible. By assumption, $\textup{def}(Q)\ge 1$. Let $Q=v_1\dots v_k$ with $v_1=u$ and $v_k=v$. Then $v_iv_{i+1}\notin E(G)$ for some $i$ with $1\le i\le k-1$. 

We claim that $N_{G-V(Q)}(v_i)\cap N_{G-V(Q)}(v_{i+1})=\emptyset$. Otherwise, there exists a common neighbor, say $x$, of $v_i$ and $v_{i+1}$ outside $Q$, which follows that $Q'=v_1Qv_ixv_{i+1}Qv_k$ is an ore $(u,v)$-path which contains all the vertices in $V(P)$ with deficit degree smaller than $\textup{def}(Q)$, a contradiction.

Recall that $d(v_i)+d(v_{i+1})\ge n+1$. Then
\[
d_Q(v_i)+d_Q(v_{i+1})\ge d(v_i)+d(v_{i+1})-|V(G)\setminus V(Q)|\ge k+1.
\] 
Let $N'=\{v_{j+1}:v_iv_j\in E(G),j\le k-1\}$. We have $|N'|\ge d_Q(v_i)-1$ and $N'\cup N_Q(v_{i+1})\subseteq V(Q)\setminus\{v_{i+1}\}$. It then follows by principle of inclusion-exclusion that 
\[
|N'\cap N_Q(v_{i+1})|\ge |N'|+d_Q(v_{i+1})-(k-1)\ge d_Q(v_i)-1+d_Q(v_{i+1})-(k-1)\ge 1,
\]
which shows that there exists an integer $j$ with $1\le j\le k-1$ such that $v_{j}v_i,v_{j+1}v_{i+1}\in E(G)$. 
Then either $j\le i-2$ or $j\ge i+2$ and so either \[
v_1Pv_jv_iPv_{j+1}v_{i+1}Pv_k
\]
or
\[
v_1Pv_iv_jPv_{i+1}v_{j+1}Pv_k
\]
is an ore $(u,v)$-path which contains all the vertices of $P$ with deficit degree smaller than $\textup{def}(Q)$, a contradiction.
\end{proof}


Theorem~\ref{Theorem-orepath} follows from the above lemma immediately. Also, the well-known result \cite{Ore} follows.

\begin{theorem}
Let $G$ be a graph of order $n$. If $d(u)+d(v)\ge n+1$ for every non-adjacent vertex pair $u,v$, then $G$ is hamiltonian-connected.
\end{theorem} 

Now, we are ready to prove Theorem \ref{Theorem-path}.

\begin{proof}[Proof of Theorem \ref{Theorem-path}]
If $n-\omega+1\ge \frac{n+1}{2}$, then the result follows by Theorem \ref{Theorem-orepath}. Suppose in the following that $n-\omega+1< \frac{n+1}{2}$, that is, $\omega>\frac{n+1}{2}$. It is trivial if $\omega=n$. So we may assume in the following that $\frac{n+2}{2}\le \omega\le n-1$.

Suppose to the contrary that there is no such $(u,v)$-path. Let $P$ be a $(u,v)$-path of $G$ such that $P$ contains vertices of degree at least $n-\omega+1$ as many as possible. 
Let $R=V(G)\setminus V(P)$ and $K$ be a largest clique in $G$.

 
\begin{claim}\label{newc}
There are at least two vertices of $K$ on $P$.
\end{claim}
\begin{proof}
It is trivial if $u,v\in K$. Assume in the following that $u\notin K$.
Then we have \[
|N_K(u)|=|N(u)|+|K|-|N(u)\cup K|\ge n-\omega+1+\omega-(n-1)=2.
\]
If $v\notin K$, then $|N_K(v)|\ge 2$. Let $z_1$ and $z_2$ be two distinct vertices such that $z_1\in N_K(u)$ and $z_2\in N_K(v)$. Then $uz_1x_1\dots x_{\omega-2}z_2v$ is a $(u,v)$-path containing vertices of degree at least $n-\omega+1$, where $x_1,\dots,x_{\omega-2}$ are vertices of $K\setminus\{z_1,z_2\}$. 
If $v\in K$, then $uy_1y_2\dots y_{\omega-1}v$ is a $(u,v)$-path containing vertices of degree at least $n-\omega+1$, where $y_1\in N_K(u)\setminus\{v\}$, $y_2,\dots,y_{\omega-1}$ are vertices of $K\setminus\{y_1,v\}$. This shows that 
$|V(P)|\ge \omega+1$ in both cases by the assumption of $P$. As $\omega\ge \frac{n+2}{2}$, we have \[
|V(P)\cap K|=|V(P)|+|K|-|V(P)\cup K|\ge \omega+1+\omega-n\ge 3,
\]
as desired.
\end{proof}

\begin{claim}\label{newcc}
For any vertex $x$ outside $P$, $N_P^+(x)$ and $N_P^-(x)$ are independent. 
\end{claim}
\begin{proof}
To the contrary, assume that $z_1^+z_2^+\in E(G)$ with $z_1,z_2\in N_P(x)$ in order on $P$, then 
\[
uPz_1xz_2Pz_1^+z_2^+Pv
\]
is a $(u,v)$-path containing more vertices of degree at least $n-\omega+1$ in $G$, a contradiction. This shows that $N_P^+(x)$ is independent. Similarly, we have $N_P^-(x)$ is independent, as desired.
\end{proof}

\begin{claim}\label{c2}
Each vertex in $K$ is contained in $P$, or $G\cong K_{ \frac{n}{2}}\vee \overline{K_{\frac{n}{2}}}$ or $K_2\vee 2K_{\frac{n}{2}-1}$.
\end{claim}
\begin{proof}
Suppose to the contrary that there exists some vertex in $K$ that is not contained in $P$. Let $\omega_1=|K\cap R|$ and $\omega_2=|K\cap V(P)|$. Evidently, $\omega_1+\omega_2=\omega$ with $\omega_1,\omega_2\ge 1$. Moreover, by Claim \ref{newc}, we have $\omega_2\ge 2$.
 
By the fact that $\omega\ge \frac{n+2}{2}$ and the assumption of $P$, for any two vertices of $K$ on $P$, the segment between these two vertices has at least $\omega_1$ vertices outside $K$. Then 
\begin{equation}\label{eqc1}
|V(P)|\ge \omega_1(\omega_2-1)+\omega_2
\end{equation}
with equality only if $u,v\in K$.
So
\begin{equation}\label{eqc2}
n\ge |V(P)|+|K\cap R|\ge \omega_1(\omega_2-1)+\omega_2+\omega_1\ge 2\omega-2\ge n.
\end{equation}
Therefore, $n$ is even and $\omega=\frac{n}{2}+1$. Moreover,
either $\omega_1=1$, $\omega_2=\omega-1$ or $\omega_1=\omega-2$, $\omega_2=2$. 

Suppose that $\omega_1=\omega-2$ and $\omega_2=2$. We have from Eq.~\eqref{eqc1} that $u,v\in K$ and $|V(P)|=\frac{n}{2}+1$. It then follows that $R=K\setminus \{u,v\}$. If there exists an edge between $V(P)\setminus \{u,v\}$ and $R$, then we can obtain a $(u,v)$-path containing more vertices of degree at least $n-\omega+1$, a contradiction. So there is no edge between $V(P)\setminus \{u,v\}$ and $R$. If $G[V(P)]$ is not a clique, then there is some vertex on $P$ of degree less than $n-\omega+1$ and so the $(u,v)$-path containing all vertices of $K$ contains more vertices of degree at least $n-\omega+1$, a contradiction. So $G[P]$ is a clique, this shows that $G\cong K_2\vee 2K_{\frac{n}{2}-1}$.

Suppose that $\omega_1=1$ and $\omega_2=\omega-1$. 
Denote by $x$ the unique vertex in $K\setminus V(P)$.
We have from Eq.~\eqref{eqc1} that $|V(P)|=2\omega_2-1$ and $u,v\in K$.
By Claim~\ref{newcc},
 $N_P^+(x)$ is independent and so $N_P^+(x)\cap K=\emptyset$ and $V(P)=N_P(x)\cup N_P^+(x)$. 
Also, by Eq.~\eqref{eqc2}, we have $R=K\setminus V(P)$ and hence $K=N_P(x)\cup \{x\}$. 
We claim that each vertex in $N_P^+(x)$ has degree at least $n-\omega+1$. To the contrary, suppose that $y\in N_P^+(x)$ has degree at most $n-\omega$. Then $uPy^-xy^+Pv$ is a $(u,v)$-path containing more vertices of degree at least $n-\omega+1$ in $G$, a contradiction.
Thus each vertex in $N_P^+(x)$ is adjacent to any vertex in $N_P(x)$ and so $G\cong K_{\frac{n}{2}}\vee \overline{K_{\frac{n}{2}}}$.
\end{proof}

If $G\cong K_{\frac{n}{2}}\vee \overline{K_{\frac{n}{2}}}$ or $G\cong K_2\vee 2K_{\frac{n}{2}-1}$, then $G$ does not contain a $(u,v)$-path satisfying the condition. Suppose in the following that $G\ncong K_{\frac{n}{2}}\vee \overline{K_{\frac{n}{2}}}$ and $G\ncong K_2\vee 2K_{\frac{n}{2}-1}$. 
By Claim~\ref{c2}, each vertex in $K$ is contained in $P$. 
Let $x$ be a vertex of degree at least $n-\omega+1$ outside $P$. 
   
Since $N_P^+(x)$ is independent by Claim~\ref{newcc}, $|N_P^+(x)\cap K|\le 1$ and so
\begin{align*}
n\ge |V(P)|+|N_R(x)|+|\{x\}|&\ge |K\setminus N_P^+(x)|+|N_P^+(x)|+|N_R(x)|+1\\
&\ge \omega-1+|N_P(x)|+|N_R(x)|\ge \omega-1+n-\omega+1=n.
\end{align*}
As the above inequality holds with equality, we have $|N_P^+(x)\cap K|= 1$, $V(P)=K\cup N_P^+(x)$, and $V(G)=V(P)\cup N_R(x)\cup \{x\}$. 
Moreover, $|N_P^+(x)|=d_P(x)-1$, which shows that $v\in N_P(x)$. Similarly, we have $|N_P^-(x)\cap K|= 1$, $V(P)=K\cup N_P^-(x)$, and $|N_P^-(x)|=d_P(x)-1$, which shows that $u\in N_P(x)$. 
Therefore, $N_P^+(x)\setminus K=N_P^-(x)\setminus K$ and so
\[
|N_P^+(x)\cap N_P^-(x)|=d_P(x)-1 \mbox{ or } |N_P^+(x)\cap N_P^-(x)|=d_P(x)-2.
\]
Denote by $x_1,\dots,x_t$ the neighbors of $x$ on $P$ in order with $x_1=u$ and $x_t=v$, where $t=d_P(x)$.

If $|N_P^+(x)\cap N_P^-(x)|=d_P(x)-1$, then $N_P^+(x)=N_P^-(x)$ and so $t=\omega-1$. This shows that $x_i^+=x_{i+1}^-$ for $i=1,\dots,\omega-2$. 
Since $\omega\ge \frac{n+2}{2}$, we have $\omega=\frac{n+2}{2}$ and $R=\{x\}$. 
If $d(x_i^+)<\frac{n}{2}$ for some $i=1,\dots,\frac{n}{2}-1$, then \[
uPx_ixx_{i+1}Pv
\]
is a $(u,v)$-path containing more vertices of degree at least $\frac{n}{2}$, a contradiction. So $d(x_i^+)\ge \frac{n}{2}$ for each $i=1,\dots,\frac{n}{2}-1$. Recall that $V(G)=V(P)\cup \{x\}=N_P(x)\cup N_P^+(x)\cup \{x\}$ and $N_P^+(x)$ is independent. We have $N(x_i^+)=N_P(x)$ and so $G\cong K_{\frac{n}{2}}\vee \overline{K_{\frac{n}{2}}}$, a contradiction.

Suppose next that $|N_P^+(x)\cap N_P^-(x)|=d_P(x)-2$. 

If $d_P(x)=2$, then as $d(x)\ge n-\omega+1$, we have $d_R(x)\ge n-\omega-1$. Note that there are at most $n-\omega$ vertices outside $P$ by Claim \ref{c2}. So $|R|=n-\omega$, $R=N_R(x)\cup \{x\}$ and $V(P)=K$. If there is some vertex $z\in N_R(x)$ adjacent to some vertex $z'\in K\setminus\{u,v\}$, then, with $z_1,\dots,z_{\omega-3}$ being vetices of $K\setminus \{u,v,z'\}$, $uxzz'z_1\dots z_{\omega-3}v$ is a $(u,v)$-path containing more vertices of degree at least $n-\omega+1$, a contradiction. So each vertex in $N_R(x)$ is not adjacent to any vertex in $K\setminus\{u,v\}$, that is, $N_P(z)\subseteq\{u,v\}$ for any $z\in N_R(x)$. This shows that $G\in \mathcal{H}(n,\omega)$.

Assume in the following that $t=d_P(x)\ge 3$.
Recall that $|N_P^+(x)\cap K|=1$, say $x_q^+\in K$ with $1\le q\le t-1$. Then $N_P^+(x)\setminus \{x_q^+\}=N_P^-(x)\setminus K\neq \emptyset$, which shows that $x_{i}^+=x_{i+1}^-$ for each $i=1,\dots,t-1$ with $i \neq q$ and each vertex in the segment $x_q^+\overrightarrow{P}x_{q+1}^-$ belongs to $K$. 
We claim that $d(x_i^+)\ge n-\omega+1$ for each $i=1,\dots,t-1$ with $i \neq q$, as otherwise, say $d(x_j^+)\le n-\omega$ for some $j=1,\dots,t-1$ with $j\ne q$, then 
\[
uPx_jxx_{j+1}Pv
\]
is a $(u,v)$-path with more vertices of degree at least $n-\omega+1$, a contradiction. 
If $N_R(x_j^+)\ne \emptyset$ for some $j=1,\dots,t-1$ with $j \neq q$, say $y\in N_R(x_j^+)$, then as $R=N_R(x)\cup \{x\}$, $xy\in E(G)$ and so \[
uPx_j^+yxx_{j+1}Pv
\]
is a $(u,v)$-path containing more vertices of degree at least $n-\omega+1$, a contradiction. So $N_R(x_i^+)=\emptyset$ for each $i=1,\dots,t-1$ with $i \neq q$.
If $N(x_j^+)\cap (K\setminus\{x_1,\dots,x_t\})\ne \emptyset$ for some $j=1,\dots,t-1$ with $j\ne q$, say $x_j^+z\in E(G)$ with $z\in K\setminus\{x_1,\dots,x_t\}$, then as $N_P^+(x)$ and $N_P^-(x)$ are independent, $z\ne x_i^+$ for $i=1,\dots,t-1$ with $i\ne q$ and so $z$ lies in the segment $x_q^+\overrightarrow{P}x_{q+1}^-$ with $z\ne x_q^+,x_{q+1}^-$. So $z^-,z^+\in K$. Recall that $x_q^+,x_{q+1}^-\in K$. Then either $j\le q-1$ or $j\ge q+1$ and so either
\[
uPx_j^+zPx_{q+1}^-z^-Px_{j+1}xx_{q+1}Pv
\]
or
\[
uPx_qxx_jPz^+x_q^+Pzx_j^+Pv
\]
is a $(u,v)$-path containing more vertices of degree at least $n-\omega+1$, a contradiction. 
Recall that $V(G)=V(P)\cup N_R(x)\cup \{x\}$ and $N_R(x_i^+)=\emptyset$ for $i=1,\dots,t-1$ with $i\ne q$. We have $N(x_i^+)\subseteq N_P(x)$ for $i=1,\dots,t-1$ with $i\ne q$. Recall that $d_P(x_i^+)\ge n-\omega+1$. So $t\ge n-\omega+1$. It then follows from 
\[
n\ge  |V(P)|+|\{x\}|=|K\setminus N_P^+(x)|+|N_P^+(x)|+1\ge \omega-1+t-1+1\ge n
\]
that $t=n-\omega+1$, $R=\{x\}$ and $N(x_i^+)=N(x)=\{x_1,\dots,x_t\}$. This shows that $G\cong K_{n-\omega+1}\vee (K_{2\omega-n-1}\cup \overline{K_{n-\omega}})$. 
\end{proof}

\section*{Acknowledgement}  The authors are grateful to Professor Xingzhi Zhan and Professor Bo Zhou for their constant support and guidance. The research of Li was supported by the NSFC grant 12271170 and Science and Technology Commission of Shanghai Municipality (STCSM) grant 22DZ2229014. 

\end{document}